\documentclass[a4paper,10pt]{amsart}
\usepackage{amssymb}
\usepackage{amsmath}
\usepackage{amsthm}

\theoremstyle{plain}

\newtheorem{theorem}[subsection]{Theorem}

\newtheorem{lemma}[subsection]{Lemma}

\theoremstyle{definition}

\newcommand{\Z}{\mathbb Z}

\newcommand{\R}{\mathbb{R}}

\title{A short proof of a sharp Weyl law for the special orthogonal group}

\author[F. Chamizo]{Fernando Chamizo}
\address[F. Chamizo]{Department of Mathematics, Universidad Aut\'{o}noma de Madrid and ICMAT, 28049 Madrid, Spain}
\email{fernando.chamizo@uam.es}

\author[J. Granados]{Jos\'{e} Granados}
\address[J. Granados]{
Instituto de Ciencias Matem\'{a}ticas, Consejo Superior de Investigaciones Cient\'{\i}ficas, Nicol\'{a}s Cabrera 13-15, 28059 Madrid, Spain}
\email{jose.granados@icmat.es}
\thanks{The first author is partially supported by the grant  MTM2017-83496-P 
from
the Spanish Ministry of Economy and Competitiveness and through the ``Severo Ochoa Programme for Centres of Excellence in R{\&}D'' (SEV-2015-0554).
The
second author is supported by a PhD fellowship at the Instituto de Ciencias Matem\'{a}ticas.}

\keywords{Weyl's law, special orthogonal group, modular forms}

\subjclass[2010]{58C40, 22C05, 11F30}

\begin{document}

\begin{abstract}
We give a short proof of a strong form of Weyl's law for $\text{SO}(N)$ using well known facts of the theory of modular forms. The exponent of the error term is sharp 
when the rank is at least~$4$. We also discuss the cases with smaller rank improving previous results.  
\end{abstract}

\maketitle


\section{Introduction}

If $M$ is a compact Riemannian manifold of dimension $d$ then the spectrum of Laplace-Beltrami operator $-\Delta$ on $M$ is discrete and corresponding to nonnegative eigenvalues. Let $\mathcal{N}(\lambda)$ be the counting function, namely the cardinality of the eigenvalues less or equal than $\lambda$ counted with multiplicity. Weyl's law states the asymptotic behavior
\begin{equation}\label{wlaw}
 \lim_{\lambda\to\infty}
 \lambda^{-d/2}
 \mathcal{N}(\lambda)
 =
 \frac{2\mathop{\text{Vol}}(M)}{d(4\pi)^{d/2}\Gamma(d/2)}
\end{equation}
where $\mathop{\text{Vol}}(M)=\int_M\omega$ with $\omega$ the Riemannian volume form. 
It comes from early works by Weyl related to mathematical physics but it seems that in this generality the first proof was not given until~1949 \cite{MiPl}. The study of sharper asymptotics in this setting or allowing boundaries has revealed a deep connection between analysis, geometry and mathematical physics \cite{ivrii}. In this interplay,  arithmetic has not stood aside \cite{hejhal} \cite{colin} \cite{skriganov} and it has a main role in this paper.

We use very basic properties of modular forms. Recall that they are holomorphic functions on the upper half plane with a Fourier expansion $\sum_{k=0}^\infty a_ke^{2\pi ikz}$. They have, in some sense, symmetries given by a finite index subgroup of $\text{SL}_2(\Z)$. In particular, if $f$ is a modular form of weight $r$ the function $|\Im z|^{r/2}|f(z)|$ is invariant under the transformations of the group and if it is a so-called cusp form, it remains bounded. An elementary argument \cite[Lem.\;3.2]{chamizo_a}
\cite[Th.\;5.3]{iwaniec} proves
\begin{equation}\label{auto_b}
\sum_{k\le K} a_k=O\big(K^{r/2}\log K\big)
\qquad\text{and}\qquad
C^{-1}<K^{-r}\sum_{k\le K} |a_k|^2<C
\end{equation}
for $K$ large and  certain $C>0$ depending on $f$. 
The connection with modular forms we exploit in this work, is that 
when $P$ is a harmonic polynomial of degree $\nu>0$ the function
\begin{equation}\label{thetaP}
 \Theta_P(z)
 =
 \sum_{k=1}^\infty
 k^{\nu/2}
 r_n(k,P)e^{2\pi ikz}
 \qquad
 \text{where}
 \quad
 r_n(k,P)
 =
 \sum_{\|\vec{m}\|^2=k}
 P\Big(\frac{\vec{m}}{\sqrt{k}}\Big)
\end{equation}
(here it is assumed $\vec{m}\in\Z^n$) 
is a cusp form  of weight $\nu+n/2$. The proof is essentially an involved application of the Poisson summation formula \cite[Th.\;10.9]{iwaniec} being the hardest part the computation of the multiplier that is irrelevant for \eqref{auto_b}.

\medskip

Our goal is to provide a short proof of  

\begin{theorem}\label{mainth}
Let $M=\text{\rm SO}(N)$, $n$ its rank  and $C_d$ the right hand side of \eqref{wlaw}. Then
\[
 \mathcal{N}(\lambda)
 -
 C_d\lambda^{d/2}
 =
 \begin{cases}
  O\big(\lambda^{d/2-1}\big) &\text{if }n>4,
  \\
  O\big(\lambda^{d/2-1}\log\lambda\big) &\text{if }n=4.
 \end{cases}
\]
\end{theorem}

Recall that the rank of $\text{\rm SO}(N)$ is  $n=N/2$ for even $N$ and $n=(N-1)/2$ for odd $N$.
As explained later, $\log \lambda$ can be replaced by $(\log \lambda)^\alpha$ for some $\alpha<1$ but it would require non-elementary techniques that would contrast with the simple arguments employed in our proof in the next section. We think that our approach, with non-essential modifications, may also cover the cases $\text{U}(N)$ and the classical groups (in the tables \cite[VIII.1,IX.8]{simon}).

For $n>4$ in principle this result  follows from \cite[Th.\;5.1]{MoTa} but we think that there is a gap in the proof provided there related with the equidistribution of lattice points on spheres. It can be quantified with homogeneous polynomials $P$ in $n$ variables through
\begin{equation}\label{En}
 E_n(k,P)=
 \frac{r_n(k,P)}{r_n(k)}
 -
 \int_{S^{n-1}}
 P\widetilde{\omega}
\end{equation}
where we have abbreviated, as usual,
$r_n(k)=r_n(k,1)$ and $\widetilde{\omega}$ is the normalized volume form of~$S^{n-1}$, the standard one divided by $\mathop{\text{Vol}}(S^{n-1})$.
In \cite{MoTa} it is claimed $E_n(k,P)=O\big(k^{-(n-1)/4}\big)$.
Using that for $n>4$ it is known \cite[Cor.\;11.3]{iwaniec} that $r_n(k)<C_nk^{n/2-1}$, it would imply when $P$ is harmonic and non-constant that 
$\sum_{k\le K} \big(k^{\nu/2}r_n(k,P)\big)^2=O\big(K^{\nu+(n-1)/2}\big)$, and this contradicts the second formula of \eqref{auto_b} and finer estimates based on the Rankin-Selberg convolution. 
As the matter of fact $E_3(k,P)=o(1)$, the so-called ``Linnik problem'', 
remained as a conjecture during many years. Finally it was proved in 1988 by W.~Duke \cite{duke} using a breakthrough due to H.~Iwaniec \cite{iwaniec_h} and still the best known result is $E_3(k,P)=O_\alpha\big(k^{-\alpha}\big)$ for any $\alpha<1/28$. 

\section{Proof of Theorem~\ref{mainth}}

After the fundamental work of Weyl on the 
theory of compact Lie groups, the irreducible representations are determined by their highest weights \cite{sepanski}.
On the other hand, we know that the entries of the matrices of the irreducible representations are the eigenfunctions of the $-\Delta$ that is the differential form of the quadratic Casimir operator \cite[VIII.3]{knapp}, \cite[\S32]{taylor}. The outcome is that the eigenvalues are given by $\|\mu+\rho\|^2-\|\rho\|^2$ for $\mu$ in certain sectors of a lattice and $\rho$ a constant vector (see   \cite{fegan} for generalizations).
Moreover, the multiplicities are given by Weyl's dimension formula  \cite[\S7.3.4]{sepanski}. In the case of  $\text{SO}(N)$ the eigenvalues are indexed by $\vec{b}\in\Z^n$
where $n$ is the rank, 
with $b_1\ge b_2\ge\dots\ge |b_n|$ 
for $N=2n$ and  
with $b_1\ge b_2\ge\dots\ge b_n\ge 0$ 
for $N=2n+1$.  
In the even case, writing $x_j=b_j+n-j$, for each choice of 
$\vec{b}$ we have the eigenvalue  $\lambda_{\vec{b}}$ and its multiplicity $m$
\begin{equation}\label{eig_m}
\lambda_{\vec{b}}=
\sum_{j=1}^n x_j^2
-
\frac{n(n-1)(2n-1)}{6}
\qquad\text{and}\qquad
m
=
\Bigg(
\prod_{i=0}^{n-1}
\prod_{j=i+1}^{n-1}
\frac{x_{n-j}^2-x_{n-i}^2}{j^2-i^2}
\Bigg)^2.
\end{equation}
Similarly, in the odd case, writing $x_j=2b_j+2n-2j+1$, the formulas are 
\begin{equation}\label{eig_mo}
\lambda_{\vec{b}}=
\frac 14\sum_{j=1}^n x_j^2
-
\frac{n(4n^2-1)}{12}
\quad\text{and}\ \ 
m
=
\Bigg(
\prod_{i=0}^{n-1}
\frac{x_{n-i}}{2i+1}
\prod_{j=i+1}^{n-1}
\frac{x_{n-j}^2-x_{n-i}^2}{(2j+1)^2-(2i+1)^2}
\Bigg)^2.
\end{equation}

The relation with lattice points problems is given through the following result \cite{MoTa}. We provide a proof not appealing to the properties of the underlying Lie algebra. 
\begin{lemma}\label{lsum}
Consider $m$ in \eqref{eig_m} and \eqref{eig_mo}  as a polynomial in the $x_1,\dots, x_n$ and let $\chi_R$ the characteristic function of the ball of radius $R$ in $\R^n$. Then  for $N$ even
\[
 \mathcal{N}(\lambda)
=
\frac{1}{2^{n-1}n!}
\sum_{\vec{x}\in \Z^n}
m(\vec{x})
\chi_R(\vec{x})
\qquad\text{where}
\quad
R^2=\lambda 
+
\frac{n(n-1)(2n-1)}{6},
\]
and 
for $N$ odd,
\[
 \mathcal{N}(\lambda)
=
\frac{1}{2^{n}n!}
\sum_{\vec{x}\in \mathbb{O}^n}
m(\vec{x})
\chi_R(\vec{x})
\qquad\text{where}
\quad
R^2=4\lambda 
+\frac{n(4n^2-1)}{3}
\]
where $\mathbb{O}$ denotes the odd integers.
\end{lemma}

\begin{proof}
Due to the relationship between $b_j$'s and $x_j$'s, the eigenvalues of $\text{SO}(N)$ can also be indexed by the elements of the set $C_E=\left\{\vec{x}\in\Z^n\;:\;x_1>x_2>\dots>|x_n|\right\}$ if $N$ is even, and by those of $C_O=\left\{\vec{x}\in \mathbb{O}^n\;:\;x_1>x_2>\dots>x_n>0\right\}$ if $N$ is odd. As $m(\vec{x})$ is invariant under ordering and sign changes in the $x_j$'s in both cases, we can consider for $N$ even all $n$-tuples in $\Z^n$ and then divide by the $2^{n-1}n!$ posibilities for those changes on $C_E$; and for $N$ odd all $n$-tuples in $\mathbb{O}^n$ and then divide by the $2^nn!$ posibilities for those changes on $C_O$. Thus, we obtain the expression for $\mathcal{N}(\lambda)$, just by noting that $\lambda_{\vec{b}}\leq\lambda$ implies $\|\vec{x}\|^2\le R^2$ in each case.
\end{proof}

After these preliminary considerations, we can deduce Theorem~\ref{mainth} for even $N$ in few lines from \eqref{auto_b} and a classic and basic result about an arithmetic function.

Each homogeneous polynomial $P$ of degree $g$ can be written uniquely as \cite{iwaniec}
\[
 P(\vec{x}) = \sum_{1\le l\le g/2}
 \|\vec{x}\|^{2l}P_{g-2l}(\vec{x})
 \qquad\text{with $P_j$ a harmonic polynomial of degree $j$}. 
\]
When we apply this to $m(\vec{x})$ that has even degree $d-n$, with $d=\dim\big(\text{SO}(2n)\big)$, separating the contribution of the constant harmonic polynomial, 
we deduce from Lemma~\ref{lsum} that there exists~$C$ and some harmonic polynomial $P$ of degree $0<\nu\le d-n$ such that 
\begin{equation}\label{nnlatt}
 \mathcal{N}(\lambda)
=
C
\sum_{\vec{x}\in \Z^n}
\|\vec{x}\|^{d-n}
\chi_R(\vec{x})
+O\Big(
\sum_{\vec{x}\in \Z^n}
\|\vec{x}\|^{d-n-\nu}
P(\vec{x})
\chi_R(\vec{x})
\Big).
\end{equation}
Grouping together the terms with $\|\vec{x}\|^2=k$, we have, with the notation as in \eqref{thetaP},
\begin{equation}\label{nlatt}
 \mathcal{N}(\lambda)
=
C
\sum_{k\le R^2}
k^{(d-n)/2}r_n(k)
+O\Big(
\sum_{k\le R^2}
k^{(d-n)/2}
r_n(k,P)
\Big).
\end{equation}
By the first formula in \eqref{auto_b} the $O$-term is 
$O\big(R^{d-n/2}\log R\big)$. On the other hand, a classic result states that the average behavior of the  arithmetic function $r_n(k)$ is
\begin{equation}\label{av_rn}
\sum_{k\le R^2}
r_4(k)
= C_4 R^4 +O(R^2\log R)
\ \text{ and }\ 
\sum_{k\le R^2}
r_n(k)
= C_n R^n +O(R^{n-2})
\quad\text{for $n>4$},
\end{equation}
where $C_n$ is the volume of the unit ball. 
The proof is a simple partial summation from Jacobi's formula for $r_4(n)$ \cite[p.\;22]{IwKo} and the simplest exponential sum method to avoid the logarithm when passing from $n=4$ to $5$ (for the details, see for instance \cite[\S15,\;Satz\;2]{fricker}). In fact, there are also elementary proofs of the formula for $r_4(n)$ \cite{venkov}
and of the exponential sums estimation \cite[Th.\;2.2]{GrKo}. 

When we substitute in \eqref{nlatt} the modular form estimate and the average result for $r_n(k)$, we conclude Theorem~\ref{mainth} with an unidentified constant that has to be $C_d$ to match \eqref{wlaw}.

\

The odd case  follows with some technical modifications. In this case one obtains 
\eqref{nnlatt} with~$\mathbb{Z}$ replaced by $\mathbb{O}$. The first sums gives readily the first sum in \eqref{nlatt} with $r_n(k)$ replaced by $r_n^*(k)$, defined as the number of representations as a sum of $n$ odd squares.  In connection with this, \eqref{av_rn} still holds (with different constants) when $r_n$ is replaced by $r_n^*$ because there is an analog for $r_4^*$ of Jacobi's formula (that it is indeed simpler \cite{carlitz}). 

On the other hand, given $B\subset\{1,2,\dots, n\}$ define $T_B$ as the diagonal matrix $\text{diag}(a_1,\dots, a_n)$ where $a_j=2$ if $j\in B$ and $a_j=1$ otherwise. By the inclusion-exclusion principle
\[
 \sum_{\vec{x}\in\mathbb{O}^n}
 \|\vec{x}\|^{d-n-\nu}
 P(\vec{x})
 \chi_R(\vec{x})
 =
 \sum_{B\subseteq\{1,\dots, n\}}
 (-1)^{|B|}
 \sum_{\vec{x}\in\mathbb{Z}^n}
\|T_B\vec{x}\|^{d-n-\nu}
 P_B(\vec{x})\chi_R(T_B\vec{x})
\]
with $P_B=P\circ T_B$.
Then the $O$-term in \eqref{nlatt} holds replacing $r_n(k,P)$ by
\[
 \quad
 r_n^*(k,P_B)
 =
 \sum_{Q(\vec{m})=k}
 P_B\Big(\frac{\vec{m}}{\sqrt{k}}\Big)
 \qquad\text{with }
 Q(\vec{x})=\|T_B\vec{x}\|^2.
\]
The polynomial $P_B$ is a spherical function with respect to the quadratic form $Q$ and the theory  assures that $k^{\nu/2}r_n^*(k,P_B)$ are the Fourier coefficients of a cusp form of weight $\nu+n/2$ \cite{iwaniec} and hence the same bound as in the even case applies.

\section{Some remarks about the true order of the error term}

A natural question is if the error term in Theorem~\ref{mainth} is sharp. The general result of \cite{DuGu} suggests that the exponent of $\lambda$ cannot be lowered. We give in this section a closer view taking advantage of the arithmetic interpretation. 

\medskip

As we saw in the previous section, when we substitute the first formula of \eqref{auto_b} in \eqref{nlatt} we get for even $N$ 
\begin{equation}\label{subs_m}
 \mathcal{N}(\lambda)
=
C
\sum_{k\le R^2}
k^{(d-n)/2}r_n(k)
+O\Big(
R^{d-n/2}\log R
\Big)
\end{equation}
and a similar formula for odd $N$ replacing $r_n$ by $r_n^*$. 

On the other hand, it is known for $n>4$ that  $r_n(k)>C_n k^{n/2-1}$ and $r_n^*(k)>C_n k^{n/2-1}$, under $4\mid k-n$,
for certain constant $C_n>0$ \cite[Cor.\;11.3]{iwaniec} \cite{carlitz}. 
It implies that when $R^2$ reaches an integer, $\mathcal{N}(\lambda)$ increases by an amount comparable to $R^{d-2}$. In particular, in Theorem~\ref{mainth} it is neither possible to change 
$O\big(\lambda^{d/2-1}\big)$
by
$o\big(\lambda^{d/2-1}\big)$
nor to extract a smooth secondary main term (even if $\lambda$ is restricted to integral values $r_n(k)$ \cite[\S11.5]{iwaniec}).

The case $n=4$ is more involved. Without entering into details, adapting the folklore techniques reviewed in \cite{HaIv} it is possible to go beyond the first bound in \eqref{auto_b} to get
\[
 \mathcal{N}(\lambda)
=
C
\sum_{k\le R^2}
k^{(d-4)/2}r_4(k)
+O\Big(
R^{d-\alpha_0}\log R
\Big)
\qquad\text{for certain $\alpha_0>2$}. 
\]
When $k$ is the product the the first odd primes, from the formula  for $r_4(k)$ and Mertens formula, it follows that   $r_4(k)\sim \frac{48e^{\gamma}}{\pi^2}k\log\log k$.
As before, one concludes as before that $\mathcal{N}(\lambda)$ is not
$o\big(\lambda^{d/2-1}\log\log\lambda\big)$. 
On the other hand, 
in \eqref{av_rn} $\log R$ can be lowered to $(\log R)^{2/3}$ with advanced methods due to N.M.~Korobov and I.M.~Vinogradov \cite{walfisz2}.
In Weyl's law it translates into the error term
$O\big(\lambda^{d/2-1}(\log\lambda)^{2/3}\big)$ for $n=4$.
The gap between $(\log R)^{2/3}$ and the barrier $\log\log R$ in the error term for the average of $r_4(n)$ has remained unchanged during the last 50 years
(see \cite{IvKr} for more information).
The same uncertainty applies to the average of $r_4^*(n)$.

\section{The cases $n<4$}

We now discuss the low dimensional cases that are not covered by the sharp estimates in \cite{MoTa}. 

For $n=1$, $\text{SO}(2)$ is just $S^1$ and the eigenfunctions  correspond to the standard orthonormal system $\{ e^{ikx}\}_{k\in\Z}$ for Fourier series. Then we have plainly 
\[
 \mathcal{N}(\lambda)
 =
 \#\big\{k\in\Z\;:\;k^2\le \lambda\big\}
 =2\lambda^{1/2}+O(1)
\]
and the $O(1)$ is sharp because $\mathcal{N}(\lambda)$ increases in one when $\lambda$ is a square. Note that the constant~$2$ matches the right hand side of \eqref{wlaw} for $M=S^1$.

In the odd case, for $n=1$ we have to consider $\text{SO}(3)$ with eigenvalues of the form $l(l+1)$, corresponding to the 
angular momentum operator in quantum physics, with multiplicity $(2l+1)^2$ and clearly
\[
 \mathcal{N}(\lambda)
 =
 \sum_{l(l+1)\le \lambda}
 (2l+1)^2
 =
 \frac{4}{3}\lambda^{3/2}
 +O\big(\lambda\big),
\]
that is sharp again.

It is known that $\sum_{k\le R^2}r_3(k)= \frac{4\pi}{3}R^3+O\big(R^{\alpha_3}\big)$ for certain $\alpha_3<3/2$ and the same applies to $\sum_{k\le R^2}r_3^*(k)$, with a different constant in the main term (the best known result in this direction allows to take any $\alpha_3>21/16$ \cite{heath} \cite{ChCrUb}). Then \eqref{subs_m} after partial summation gives Weyl's law for $n=3$ in the form
\[
 \mathcal{N}(\lambda)
 =
 C_d\lambda^{d/2}+O\big(\lambda^{d/2-3/4}\log \lambda\big).
\]
With the aforementioned techniques in \cite{HaIv} it seems possible to reduce the exponent
(one should proceed as indicated at the end of the paper for Maass forms using \cite{duke}, see also \cite{ChCrUb})
beating Theorem~4.1 of \cite{MoTa} for $N=6$ and $N=7$.  

In the case $n=2$ there is not equidistribution of the points $\|\vec{m}\|^2=k$ and the modular form approach seems superfluous.  
We show how to go beyond \cite{MoTa} using exponential sums. We focus on the even case. The odd one follows under the same lines using the formulas of the second part of  Lemma~\ref{lsum}.
The multiplicity in \eqref{eig_m} is given by the polynomial $m(x,y)=(x^2-y^2)^2$. By Lemma~\ref{lsum} and the symmetry 
\[
  \mathcal{N}(\lambda)
=
2\mathop{\sideset{}{'}\sum}_{0\le x\le  R/\sqrt{2}}
\
\sum_{y\in I_x}
m(x,y)
\qquad\text{where}\quad
I_x=\big(x, \sqrt{R^2-x^2}\big]
\]
and the prime in the first sum means that the  contribution of $x=0$ is halved. Using the variant of Euler-Maclaurin summation 
$\sum_{a<n\le b} f(n)
=
\psi(a)f(a)-\psi(b)f(b)+\int_a^b(f+\psi f')$
with $\psi(t)= t-\lfloor t\rfloor -1/2$ 
(cf. \cite[Th.\;4.2]{apostol}, it was stated in 1885 by N.Y.~Sonin), 
we have 
\[
  \mathcal{N}(\lambda)
=
T_1+T_2+T_3
\qquad\text{with}\quad
T_1= 
2\mathop{\sideset{}{'}\sum}_{0\le x\le  R/\sqrt{2}}
\int_{I_x} m(x,y)\; dy
\]
and
\[
T_2= 
2\hspace{-5.1pt}\mathop{\sideset{}{'}\sum}_{0\le x\le  R/\sqrt{2}}
(R^2-2x^2)^2\psi\big(\sqrt{R^2-x^2}\big)
,\quad
T_3= 
2\hspace{-5.1pt}
\mathop{\sideset{}{'}\sum}_{0\le x\le  R/\sqrt{2}}
\int_{I_x} 4y(y^2-x^2)\psi(y)\; dy.
\]
By partial integration $T_3=O(R^4)$. A new application of  Euler-Maclaurin summation  in $T_1$ gives
\[
 T_1
 =
 2\int_{0}^{R/\sqrt{2}}
 \int_x^{\sqrt{R^2-x^2}} m(x,y)\; dydx
 +
 O(R^4) 
 =
 \frac 14
 \int_{\R^2} m\chi_R
 +
 O(R^4) 
\]
which is $C_d\lambda^{d/2}+O\big(\lambda^{d/2-1}\big)$.
For $T_2$ we employ the existence 
of two finite Fourier series 
$Q^\pm(x)=\sum_{|m|\le M} a_m^\pm e^{2\pi imx}$  for each $M\in\Z^+$
such that $Q^-(x)\le \psi(x)\le Q^+(x)$
with $a_0^\pm =O\big( M^{-1}\big)$ and $a_m^\pm =O\big(  m^{-1}\big)$ (see for instance \cite{montgomery}).
Substituting in $T_2$ and applying a van der Corput exponent pair \cite{GrKo} $(\alpha,\beta)$  to the sum on $x$, we have
$T_2=O\big(M^\alpha R^{\beta+4}+ M^{-1}R^5\big)$ that is optimal taking $M\asymp R^{(1-\beta)/(1+\alpha)}$. The valid choice $\alpha=11/30$, $\beta=16/30$ gives $O\big(R^{191/41}\big)$ and then for $n=2$
\[
 \mathcal{N}(\lambda)
 =
 C_d\lambda^{d/2}+O\big(\lambda^{d/2-55/82}\big)
\]
with tiny improvements for other choices of the exponent pair. 

This slightly improves the cases $N=4$ and $N=5$ of  Theorem~4.1 in \cite{MoTa}.


\end{document}